\documentclass[12pt]{amsart}
\usepackage{amssymb}

\theoremstyle{plain}
\newtheorem{thm}{Theorem}[section]
\newtheorem{lem}[thm]{Lemma}
\newtheorem{conj}{Conjecture}
\newtheorem{cor}[thm]{Corollary}
\newtheorem{obs}[thm]{Observation}
\newtheorem{prp}[thm]{Proposition}
\newtheorem{prob}[thm]{Problem}
\newtheorem{defn}[thm]{Definition}

\newtheorem{remark}[thm]{Remark}
\newtheorem{example}[thm]{Example}

\newcommand{\bm}[1]{\mbox{\boldmath $#1$}}
\DeclareMathOperator{\sd}{sd}
\DeclareMathOperator{\A}{\bm{A}}
\DeclareMathOperator{\F}{\mathcal{F}}
\DeclareMathOperator{\col}{col}
\DeclareMathOperator{\B}{\mathcal{B}}

\title[The $\gamma$-vector of a barycentric subdivision]{The
$\bm{\gamma}$-vector of a barycentric subdivision}

\author[E. Nevo]{Eran Nevo}

\address{Department of Mathematics, Cornell University, Ithaca, NY 14853, USA}
\email{eranevo@math.cornell.edu}

\author[T. K. Petersen]{T. Kyle Petersen}

\author[B. E. Tenner]{Bridget Eileen Tenner}

\address{Department of Mathematical Sciences, DePaul University, Chicago, IL 60614, USA}
\email{tkpeters@math.depaul.edu\\bridget@math.depaul.edu}

\begin{document}

\begin{abstract}
We prove that the $\gamma$-vector of the barycentric subdivision of a simplicial sphere is the $f$-vector of a balanced simplicial complex.
The combinatorial basis for this work is the study of certain refinements of Eulerian numbers used by Brenti and Welker to describe the $h$-vector of the barycentric subdivision of a boolean complex.
\end{abstract}

\maketitle

\section{Introduction}

The present work can be motivated by Gal's conjecture, a certain strengthening of the Charney-Davis conjecture phrased in terms of the so-called \emph{$\gamma$-vector}.

\begin{conj}[{Gal \cite{Gal}}]\label{conj:Gal}
If $\Delta$ is a flag homology sphere, then $\gamma(\Delta)$ is nonnegative.
\end{conj}

This conjecture is known to hold for the order complex of a Gorenstein$^*$ poset \cite{KaruCD}, for all Coxeter complexes (see \cite{Stem}, and references therein), and for the (dual simplicial complexes of the) ``chordal nestohedra" of \cite{PRW}---a class containing the associahedron, permutahedron, and other well-studied polytopes.

If $\Delta$ has a nonnegative $\gamma$-vector, one might ask what is enumerated by the entries in the vector. In certain cases (the type $A$ Coxeter complex, for example), the $\gamma$-vector has an explicit combinatorial description.

Recently Nevo and Petersen showed in many cases that not only are the entries in the $\gamma$-vector of a flag homology sphere nonnegative, but they satisfy the Frankl-F{\"u}redi-Kalai inequalities \cite{NP}. In other words, such a $\gamma$-vector is the $f$-vector of a balanced simplicial complex. This suggests a strengthening of Gal's conjecture, given below.

\begin{conj}[Nevo and Petersen \cite{NP}]\label{conj:flagKrK}
If $\Delta$ is a flag homology sphere then $\gamma(\Delta)$ is the $f$-vector of a balanced simplicial complex.
\end{conj}

Our main result, stated below, confirms Conjecture~\ref{conj:flagKrK} in the case of the barycentric subdivision of a homology sphere.

\begin{thm}\label{thm:bary}
If $\Delta$ is a boolean complex with a nonnegative and symmetric $h$-vector, then the $\gamma$-vector of the barycentric subdivision of $\Delta$  is the $f$-vector of a balanced simplicial complex.
\end{thm}

In particular, if $\Delta$ is a homology sphere, then the Dehn-Sommerville relations guarantee that the $h$-vector is symmetric.  Additionally, the fact that $\Delta$ is Cohen-Macaulay means that its $h$-vector is nonnegative \cite{StanleyUBC}. Hence, Theorem \ref{thm:bary} implies the conclusion of Conjecture \ref{conj:flagKrK} in this case.
We remark that the result of Karu \cite{KaruCD} implies the nonnegativity of the $\gamma$-vector for barycentric subdivisions of homology spheres. However, our approach is quite different.

The paper is organized as follows. Section \ref{sec:prelim} provides key definitions. Section \ref{sec:bal} discusses balanced complexes and presents several key lemmas. In Section \ref{sec:poly} we elaborate on work of Brenti and Welker \cite{BW} to lay the combinatorial foundation for our main result. The proof of our main theorem, Theorem \ref{thm:bary}, is carried out in Section \ref{sec:main}. Finally, in Section \ref{sec:implications}, we show how our result has implications for similar results for the $h$- and $g$-vectors. In particular we derive a recent result of Murai \cite{MuraiBr} (who takes an approach similar to ours in building on \cite{BW}) that the $g$-vector of a barycentric subdivision of a homology sphere is the $f$-vector of a simplicial complex.

\section{Preliminaries}\label{sec:prelim}

The faces of a regular cell complex $\Delta$ have a well-defined partial order given by inclusion.  More precisely, if $F$ and $F'$ are open cells in $\Delta$, then $F \leq_{\Delta} F'$ if $F$ is contained in the closure of $F'$. By convention we assume that the empty face $\emptyset$ is contained in every face of higher dimension. A \emph{boolean cell complex} is a regular cell complex in which every lower interval $[\emptyset, F] = \{ G \in \Delta : G \leq F\}$ is isomorphic to a boolean lattice. For example, a simplicial complex is a boolean complex, because the lower interval of a simplex is simply the collection of subsets of the vertices of the simplex.

The \emph{barycentric subdivision} of a boolean complex $\Delta$, denoted $\sd(\Delta)$, is the abstract simplicial complex whose vertex set is identified with the nonempty faces of $\Delta$, and whose $i$-dimensional faces are strictly increasing flags of the form
\[ F_0 <_{\Delta} F_1 <_{\Delta} \cdots <_{\Delta} F_i,\]
where the $F_j$ are nonempty faces in $\Delta$.

The \emph{$f$-polynomial} of a $(d-1)$-dimensional boolean complex $\Delta$ is the generating function for the dimensions of the faces of the complex:
\[ f(\Delta;t) = \sum_{F \in \Delta} t^{1+\dim F} = \sum_{i=0}^d f_i(\Delta) t^i,\]
where the empty face $\emptyset$ has dimension $-1$, and so $f_0 = 1$. The \emph{$f$-vector}
\[ f(\Delta) = (f_0(\Delta), f_1(\Delta),\ldots,f_d(\Delta)) \]
is the sequence of coefficients of the $f$-polynomial.

The \emph{$h$-polynomial} of $\Delta$ is a transformation of the $f$-polynomial:
\[ h(\Delta;t) = (1-t)^d f(\Delta; \frac{t}{1-t}) = \sum_{i=0}^d h_i(\Delta) t^i,\]
and the \emph{$h$-vector} is the corresponding sequence of coefficients
\[h(\Delta) = (h_0(\Delta),h_1(\Delta),\ldots, h_d(\Delta)).\]
Although they contain the same information, the $h$-polynomial is often easier to work with than the $f$-polynomial. For instance, if $\Delta$ is a homology sphere, then the Dehn-Sommerville relations guarantee that the $h$-vector is symmetric; that is, $h_i = h_{d-i}$ for all $0\leq i\leq d$.

Whenever a polynomial $h(t) = \sum_{i=0}^d h_i t^i$ has symmetric integer coefficients, the polynomial has an integer expansion in the basis $\{ t^i(1+t)^{d-2i} : 0\leq i \leq d/2\}$, and we say that it has ``symmetry axis'' at degree $\lfloor d/2 \rfloor$.  In this case, if we write
$$h(t) = \sum_{i=0}^{\lfloor d/2 \rfloor}\gamma_i t^i(1+t)^{d-2i},$$
then we say $\gamma(h(t)) = (\gamma_0, \gamma_1, \ldots, \gamma_{\lfloor d/2 \rfloor})$ is the \emph{$\gamma$-vector} of $h(t)$. In particular if a $(d-1)$-dimensional complex $\Delta$ happens to have a symmetric $h$-vector, then we refer to the $\gamma$-vector of $h(\Delta;t)$ as the $\gamma$-vector of $\Delta$; that is, $\gamma(\Delta)=\gamma(h(\Delta;t))$.

Another invariant that we will study (though not until Section \ref{sec:implications}) is the \emph{$g$-vector} of $\Delta$, $g(\Delta) = (g_0, g_1,\ldots, g_{\lfloor d/2 \rfloor})$, defined by $g_0 = h_0$ and, for  $0\leq i \leq \lfloor d/2\rfloor$, \[ g_i(\Delta) = h_i(\Delta) - h_{i-1}(\Delta).\] The polynomial $g(\Delta;t) = \sum_{i=0}^{\lfloor d/2 \rfloor} g_i t^i$ is the generating function for the $g$-vector.

\section{Balanced complexes}\label{sec:bal}

A boolean complex $\Delta$ on vertex set $V$ is \emph{$d$-colorable} if there is a coloring of its vertices $c: V\to [d]$  such that for every face $F\in \Delta$, the restriction map $c: F\to [d]$ is injective; that is, every face has distinctly colored vertices. Such coloring is called a \emph{proper coloring}. If a $(d-1)$-dimensional complex $\Delta$ is $d$-colorable, then we say that $\Delta$ is \emph{balanced}.

Frankl, F{\"u}redi, and Kalai characterized the $f$-vectors of $d$-colorable simplicial complexes, for any $d$, in terms of certain upper bounds on $f_{i+1}$ in terms of $f_i$ \cite{FFK}. We call these the \emph{Frankl-F{\"u}redi-Kalai inequalities}, and a vector satisfying the inequalities with respect to $d$ is a \emph{$d$-Frankl-F{\"u}redi-Kalai vector}, or \emph{$d$-FFK-vector} for short.  We have no need for the explicit inequalities in the present work, but they can be found in \cite{FFK}.

A key ingredient in \cite{FFK} is the construction of a balanced simplicial complex with a specified $f$-vector. This relies on the use of a colored version of the idea of a \emph{compressed} simplicial complex. (See, for example, \cite[Section 8.5]{Z} for a description of compressed complexes as used in the characterization of $f$-vectors of arbitrary simplicial complexes.)

First, we define a \emph{$d$-colored $k$-subset} of $\mathbb{N}$ to be a set of $k$ positive integers such that no two of them are congruent modulo $d$. Denote by $\binom{\mathbb{N}}{k}_d$ the set of all $d$-colored $k$-subsets. The \emph{color} of $S \in \binom{\mathbb{N}}{k}_d$ is the set of remainders modulo $d$: $\col_d(S) =\{ s \bmod d : s \in S\}$. We list $d$-colored $k$-subsets in reverse lexicographic (``revlex'') order. For example, if $k=3$ and $d=4$, then we have \[ 123 < 124 <134 < 234 < 235 < 245 < 345 < \cdots.\] Notice that the sets $125, 135$, and $145$ are \emph{not} $4$-colored since $1\equiv 5 \bmod 4$.

If $f = (1,f_1,\ldots,f_d)$, we let \[ \F_d(f) = \{ F_{d,k}(j) : 0\leq k \leq d, 1 \leq j \leq f_k\},\] where $F_{d,k}(j)$ denotes the $j$th element in the revlex order on $\binom{\mathbb{N}}{k}_d$. For example, from above we see $F_{4,3}(1) = 123$, $F_{4,3}(2) = 124$, and so on.  In other words, the $(k-1)$-dimensional faces of $\F_d(f)$ are the first $f_k$ of the $d$-colored $k$-subsets of $\mathbb{N}$ in revlex order.

Half of the characterization given in \cite{FFK} states that if $f$ is a $d$-FFK-vector then $\F_d(f)$ is $d$-colorable simplicial complex (thus if $f_d \neq 0$ then $\F_d(f)$ is a balanced simplicial complex).

If $\Delta$ is simplicial complex with $f$-vector $f = (1,f_1,\ldots,f_d)$ then we define the \emph{$d$-compression of $\Delta$} to be $\F_d(f)$. If a $(d-1)$-dimensional simplicial complex $\Delta$ is balanced, its $d$-compression is a balanced simplicial complex, called a \emph{$d$-compressed complex}.

If $\Delta$ and $\Delta'$ are simplicial complexes on disjoint vertex sets, define their \emph{join} to be the complex $\Delta * \Delta' = \{ F \cup G : F \in \Delta, G \in \Delta'\}$. As a special case, the \emph{cone} over $\Delta$ is $\Delta^c = \Delta * \{ v\}$, for some vertex $v \notin \Delta$. It is clear by construction that for $k \geq 1$,
 \[ f_k(\Delta^c) = f_k(\Delta) + f_{k-1}(\Delta).\]

If $f = (1,f_1,\ldots, f_d)$ and $f'=(1,f'_1,\ldots,f'_d)$ are $d$-FFK-vectors such that $f_k \geq f'_k$ for all $k=1,\ldots,d$, then we say $f$ \emph{dominates} $f'$. Clearly if $f$ dominates $f'$ then $\F(f')$ is a subcomplex of $\F(f)$, and, by introducing a new vertex $v$ with color $d+1$, the complex \[\F(f) \cup \F(f')^c = \F(f) \cup \{ F \cup \{v\} : F \in \F(f')\}\] is a $(d+1)$-colorable $d$-dimensional complex with $f$-vector $f+(0,f')$. Applying this argument repeatedly proves the following lemma.

\begin{lem}\label{lem:cone1}
Consider vectors $f = (1,f_1,f_2,\ldots,f_d)$, $f^{(1)}=(1,f^{(1)}_1, f^{(1)}_2,\ldots,f^{(1)}_d)$, $\ldots$, $f^{(k)}=(1,f^{(k)}_1,f^{(k)}_2,\ldots,f^{(k)}_d)$, and require them all to be $d$-FFK-vectors with $f$ dominating $f^{(j)}$ for all $j$. Then \[ f+(0,f^{(1)}+\cdots +f^{(k)}) = \left(1, f_1+ k, f_2+\sum_{j=1}^k f_1^{(j)},\ldots, f_d+\sum_{j=1}^k f_{d-1}^{(j)}, \sum_{j=1}^k f_{d}^{(j)} \right) \] is a $(d+1)$-FFK-vector.
\end{lem}

Similarly, suppose $\Delta$ is a balanced $(d-1)$-dimensional complex on vertex set $V$ with a proper coloring $c:V\rightarrow [d]$, and let $\Delta'$  be a subcomplex of dimension $(d-2)$, on vertex set $V'$, such that $c(V')\neq [d]$. Let $f = f(\Delta) = (1,f_1,\ldots,f_d)$ and $f'= f(\Delta') = (1,f'_1,\ldots,f'_{d-1})$. Then by coning over $\Delta'$ with a new vertex $v$ colored by the element in $[d]\setminus c(V')$, we see that $\Delta \cup (\Delta')^c$ is a balanced $(d-1)$-dimensional complex with $f$-vector $f+(0,f')$. Of course, we can do this for any number of such subcomplexes $\Delta'$, and so we have the following lemma.

\begin{lem}\label{lem:cone2}
Let $\Delta$ be a balanced $(d-1)$-dimensional complex on vertex set $V$ with a proper coloring $c:V\rightarrow [d]$.  Furthermore, let $\Delta$ have (not necessarily distinct) subcomplexes $\Delta^{(1)}, \ldots, \Delta^{(k)}$, each of dimension at most $d-2$, on vertex sets $V^{(1)},\ldots,V^{(k)}$ respectively, such that for every $i$, we have $c(V^{(i)})\neq [d]$. Let $f =f(\Delta)$, and $f^{(1)}=f(\Delta^{(1)})$,  $\ldots, f^{(k)}=f(\Delta^{(k)})$. Then $\Delta \bigcup_{j=1}^k (\Delta^{(j)})^c$ is a balanced $(d-1)$-dimensional complex with $f$-vector \[ f+ (0,f^{(1)}+\cdots + f^{(k)}). \]
\end{lem}

Using the idea of $d$-compression, we will now exhibit a sufficient condition for the existence of a collection of subcomplexes as in Lemma \ref{lem:cone2}.

There is a natural embedding $\phi_d$ of the revlex order on $\binom{\mathbb{N}}{k}_{d-1}$ in the revlex order on $\binom{\mathbb{N}}{k}_d$, defined as follows.  For $s = (d-1)i + j \in \mathbb{N}$ with $0\leq i$ and $1\leq j\leq d-1$, let $\phi_d(s) = s+i = di+j$. Then if $S$ is a $(d-1)$-colored $k$-subset, let $\phi_d(S) = \{ \phi_d(s) : s \in S\}$. Observe that the color of $S$ is preserved under $\phi_d$; that is, $\col_{d-1}(S) = \col_d(\phi_d(S))$. Moreover, if $S \leq T$ in $\binom{\mathbb{N}}{k}_{d-1}$, then $\phi_d(S) \leq \phi_d(T)$ in $\binom{\mathbb{N}}{k}_d$.

Define the function $r_{d,k}(a)=b$ as the position that the $a$th element in  $\binom{\mathbb{N}}{k}_{d-1}$ gets mapped to in $\binom{\mathbb{N}}{k}_d$. That is, $\phi_d(F_{d-1,k}(a)) = F_{d,k}(b)$.

\begin{prp}\label{prp:ineq}
For all $a,k,d\in \mathbb{N}$ such that $k<d$, we have $r_{d,k}(a) \leq (k+1)a$.
\end{prp}

\begin{proof}
Let $b=r_{d,k}(a)$, and let $A = \{ S : S \leq F_{d,k}(b), d \notin \col_d(S)\}$ and $B = \{ T : T < F_{d,k}(b), d \in \col_d(T) \}$. Notice $A = \{ \phi_d(S) : S \leq F_{d-1,k}(a)\}$, and so $|A| = a$.

Define $\pi: B \to A$ by
\[\pi(T) = (T\setminus \{ r\}) \cup \{r'\},\]
where $r$ is the element in $T$ of the form $r = di$ and $r' = d(i-1)+j$ for the greatest $j$ such that $1\leq j\leq d-1$ and $j \notin \col_d(T)$.  Since $k \leq d-1$, there is always some color not in $\col_d(T)$. By construction, $\pi(T) < T$.  Thus $\pi(T) \in A$ and $\pi$ is well-defined. For example, with $d=5$, $\pi(\{ 1,4,5,8\}) = \{ 1,2,4,8\}$.

Now for any element $S \in A$ consider $T \in \pi^{-1}(S)$. We obtain such a $T$ by replacing $s$ in $S$ with a number $r$ such that $s <r <s+d$ and $r$ is a multiple of $d$. Since there is at most one such $r$ for each element of $S$, we have $|\pi^{-1}(S)|\leq |S|=k$.

Now, \[ r_{d,k}(a) = |A|+|B| \leq |A| + k|A| = (k+1)a,\] as desired.
\end{proof}

We now have the following corollary.

\begin{cor}
Let $f = (1,f_1,\ldots,f_{d-1})$ be a $(d-1)$-FFK-vector and $f'=(1,f'_1,\ldots,f'_d)$ a $d$-FFK-vector satisfying $(k+1)f_k \leq f'_k$ for $1\leq k \leq d-1$. Then there is a color-preserving isomorphism between $\F_{d-1}(f)$ and a subcomplex of $\F_d(f')$.
\end{cor}

\begin{proof}
By Proposition \ref{prp:ineq}, we have that $r_{d,k}(f_k) \leq (k+1)f_k \leq f'_k$.  Thus, from the definition of $d$-compressed complexes (and the function $r_{d,k}(i)$), if $S \in \F_{d-1}(f)$, then $\phi_d(S) \in \F_d(f')$. That is, $\phi_d(\F_{d-1}(f))$ is a subcomplex of $\F_d(f')$. As $\col_{d-1}(S) = \col_d(\phi_d(S))$, we have that $\F_{d-1}(f) \to \phi_d(\F_{d-1}(f))$ is a color-preserving isomorphism of balanced simplicial complexes, yielding the desired result: $\F_{d-1}(f) \cong \phi_d(\F_{d-1}(f)) \subset \F_d(f')$.
\end{proof}

Repeatedly using this fact, along with Lemma \ref{lem:cone2}, yields the following companion to Lemma \ref{lem:cone1}.

\begin{lem}\label{lem:dgood}
Let $f = (1,f_1,f_2,\ldots,f_d)$ be a $d$-FFK-vector, and $f^{(1)}=(1,f^{(1)}_1, f^{(1)}_2,\ldots,f^{(1)}_{d-1})$, $\ldots,$
$f^{(k)}=(1,f^{(k)}_1,f^{(k)}_2,\ldots,f^{(k)}_{d-1})$ be $(d-1)$-FFK-vectors such that $f_i \geq (i+1)f^{(j)}_i$ for all $i$ and all $j$. Then \[ f+(0,f^{(1)}+\cdots +f^{(k)}) = \left(1, f_1+ k, f_2+\sum_{j=1}^k f_1^{(j)},\ldots, f_d+ \sum_{j=1}^k f_{d-1}^{(j)}\right) \] is a $d$-FFK-vector.
\end{lem}

Let us now define terms for $f$-vectors satisfying the conditions of these lemmas.

\begin{defn}\label{def:good}
Let $f=(1,f_1,\ldots,f_d)$ be the $f$-vector of a $(d-1)$-dimensional balanced complex.
\begin{enumerate}
\item \textbf{$\bm{(d+1)}$-good}. As in Lemma \ref{lem:cone1}, let $g=(g_1,\ldots,g_d,g_{d+1})$ be a sum of $d$-FFK-vectors, each of which is dominated by $f$.
    Some, but not all, of these $d$-FFK-vectors may be shorter than $f$; that is, $g_{d+1}\neq 0$.
    Then we say that $(0,g)$ is \emph{$(d+1)$-good} for $f$. Note that $f+(0,g)$ is a $(d+1)$-FFK-vector.
\item \textbf{$\bm{d}$-good}. Let $g=(g_1,\ldots,g_d) = f^{(1)}+\cdots+f^{(k)}$, with $g_d \neq 0$, be a sum of $(d-1)$-FFK-vectors such that $f_i \geq (i+1)f_i^{(j)}$ for all $i$ and all $j$.
    Then we say that $(0,g)$ is \emph{$d$-good} for $f$. Note that $f+(0,g)$ is a $d$-FFK-vector.
\end{enumerate}
\end{defn}

The following is immediate from the definitions.

\begin{obs}\label{obs:dd+1}
Let $f=(1,f_1,\ldots,f_d)$ be the $f$-vector of a $(d-1)$-dimensional balanced complex.
If $(0,g)$ is $d$-good for $f$ and $(0,g')$ is $(d+1)$-good for $f$, then
$(0,g+g')$ is $(d+1)$-good for $f$. Recall that
$f+(0,g+g')$ is a $(d+1)$-FFK-vector.
\end{obs}

\section{Restricted Eulerian polynomials}\label{sec:poly}

In \cite{BW}, Brenti and Welker provide, among other things, explicit combinatorial formulas relating $h(\Delta)$ and $h(\sd(\Delta))$, which we now describe.

Let $S_n$ denote the symmetric group on $\{1,2,\ldots,n\}$.  For a permutation $w \in S_n$, the \emph{descent number} of $w$ is defined as
$$d(w) = |\{i : w(i) > w(i+1)\}|.$$
Let $A(n,i,j)$ denote the number of permutations $w \in S_n$ such that $w(1) =j$ and $d(w) = i$.

\begin{thm}[{Brenti and Welker \cite[Theorem 2.2]{BW}}]\label{thm:BW}
Let $\Delta$ be a $(n-1)$-dimensional boolean complex. Then for $0\leq i \leq n$,
\begin{equation}\label{eqn:bw}
h_i(\sd(\Delta)) = \sum_{j=0}^n A(n+1,i,j+1) h_j(\Delta).
\end{equation}
\end{thm}

We shall return to this characterization of the $h$-vector of $\sd(\Delta)$ in Section \ref{sec:main}. First we will investigate various properties of the coefficients $A(n,i,j)$ and related generating functions.

Let
$$S_{n,j} = \{ w \in S_n : w(1) = j\},$$
and define the \emph{restricted Eulerian polynomials} as the descent generating functions for these subsets of $S_n$:
\[ A_{n,j}(t) = \sum_{w \in S_{n,j} } t^{d(w)} = \sum_{i=0}^{n-1} A(n,i,j) t^i.\]
Note that the usual Eulerian polynomial is the descent generating function for all of $S_n$: \[ A_n(t) =\sum_{w\in S_n} t^{d(w)} = \sum_{j=1}^n A_{n,j}(t).\]

For example,
\begin{align*}
A_{4,1}(t) &= 1 + 4t + t^2,\\
A_{4,2}(t) &= 4t + 2t^2,\\
A_{4,3}(t) &= 2t+4t^2, \\
A_{4,4}(t) &=  t + 4t^2 + t^3,
\end{align*}
and $A_4(t) =1 + 11t + 11t^2 + t^3$.

The following observation is immediate by considering that if $w(1)=1$ there is never a descent in the first position, and if $w(1) =n$ there is always a descent in the first position.

\begin{obs}\label{obs:eul}
We have $A_{n,1}(t) = A_{n-1}(t)$ and $A_{n,n}(t) = tA_{n-1}(t)$.
\end{obs}

Tracking the effect of removing the letter $j$ from a permutation in $S_{n,j}$ yields the following recurrence for restricted Eulerian polynomials.

\begin{obs}[{Brenti and Welker \cite[Lemma 2.5(i)]{BW}}]\label{obs:rec}
We have
\[ A(n,i,j) = \sum_{k=1}^{j-1}A(n-1,i-1,k) + \sum_{k=j}^{n-1}A(n-1,i,k),\]
and thus
\[ A_{n,j}(t) = t\sum_{k=1}^{j-1} A_{n-1,k}(t) + \sum_{k=j}^{n-1} A_{n-1,k}(t).\]
\end{obs}

A standard involution on permutations in $S_n$ is given by mapping each $i$ to $n+1-i$. This involution has the effect of exchanging ascents and descents. Hence the following symmetries hold.

\begin{obs}[{Brenti and Welker \cite[Lemma 2.5(ii)]{BW}}]\label{obs:sym}
We have
\[ A(n,i,j) = A(n, n-1-i, n+1-j),\]
and thus
\[ A_{n,j}(t) = t^{n-1} A_{n,n+1-j}(1/t).\]
\end{obs}

We now define the \emph{symmetric restricted Eulerian polynomials} by lumping together classes fixed by the involution just described, namely all permutations beginning with either $j$ or $n+1-j$:
\[ \A_{n,j}(t)= \sum_{w \in S_{n,j}\cup S_{n,n+1-j}} t^{d(w)}.\]
Observe that
$$\A_{n,j}(t) = \begin{cases}
A_{n,j}(t) + A_{n,n+1-j}(t) & \mbox{if } j \neq (n+1)/2, \text{ and} \\
A_{n,j}(t) & \mbox{if } j = (n+1)/2.
\end{cases}$$

By Observation \ref{obs:sym}, the polynomials $\A_{n,j}(t)$ have symmetric coefficients, and hence a $\gamma$-vector.
Clearly $\A_{n,j}(t)$ has symmetry axis at degree $\lfloor \frac{n-1}{2} \rfloor$.
If
$$\A_{n,j}(t) = \sum_{i=0}^{\lfloor (n-1)/2 \rfloor} \gamma_i^{(n,j)} t^i(1+t)^{n-1-2i},$$ let $\gamma^{(n,j)}=(\gamma_0^{(n,j)},\gamma_1^{(n,j)},\ldots, \gamma^{(n,j)}_{\lfloor (n-1)/2 \rfloor})$ denote the corresponding $\gamma$-vector.

We will have reason to consider the following polynomials as well, where $1\leq j < (n+1)/2$:
\[\A'_{n,j}(t) =
tA_{n,j}(t) + A_{n,n+1-j}(t).\]
Note that,
by Observation~\ref{obs:sym},
the integer coefficients of $\A'_{n,j}(t)$ are symmetric, and so $\A'_{n,j}(t)$ has a $\gamma$-vector.
Clearly $\A'_{n,j}(t)$ has symmetry axis at degree $\lfloor n/2 \rfloor$.

Define \[ \gamma'^{(n,j)} = (\gamma'^{(n,j)}_0, \gamma'^{(n,j)}_1, \ldots, \gamma'^{(n,j)}_{\lfloor n/2 \rfloor})\] to be the $\gamma$-vector for the polynomial $\A'_{n,j}(t)$.

Using Observation \ref{obs:rec}, we obtain the following recurrences.

\begin{lem}\label{lem:rec}
We have the following recurrences for the $\gamma^{(n,j)}$ and $\gamma'^{(n,j)}$:

\begin{enumerate}

\item If $j = (n+1)/2$, then \[ \gamma^{(n,(n+1)/2)} = \gamma'^{(n-1,1)} + \gamma'^{(n-1,2)} + \cdots + \gamma'^{(n-1,(n-1)/2)}.\]

\item For $j < (n+1)/2$,
\[
\gamma^{(n,j)} =
2\sum_{k=1}^{j-1} \gamma'^{(n-1,k)} + \sum_{k=j}^{\lfloor n/2\rfloor } \gamma^{(n-1,k)}.\]

\item For $j < (n+1)/2$,
\[
\gamma'^{(n,j)} =
\sum_{k=1}^{j-1} \gamma'^{(n-1,k)} + 2\sum_{k=j}^{\lfloor n/2 \rfloor} (0,\gamma^{(n-1,k)}).
\]

\end{enumerate}

\end{lem}

\begin{proof}

Statement (1) follows because, for $j=(n+1)/2$ we have, by Observation \ref{obs:rec},
\begin{align*}
\A_{n,(n+1)/2} = A_{n,(n+1)/2}(t) &= t\sum_{k=1}^{(n-1)/2}A_{n-1,k}(t) + \sum_{k=(n+1)/2}^{n-1}A_{n-1,k}(t)\\
&= t\sum_{k=1}^{(n-1)/2}A_{n-1,k}(t) + \sum_{k=1}^{(n-1)/2}A_{n-1,n-k}(t)\\
&= \sum_{k=1}^{(n-1)/2}\A'_{n-1,k}(t).
\end{align*}

If $j < (n+1)/2$, then Observation \ref{obs:rec} implies
\begin{align*}
\A_{n,j}(t) &= A_{n,j}(t) + A_{n,n+1-j}(t)\\
&= t\sum_{k=1}^{j-1} A_{n-1,k}(t) + \sum_{k=j}^{n-1} A_{n-1,k}(t) + t\sum_{k=1}^{n-j} A_{n-1,k}(t) + \sum_{k=n+1-j}^{n-1} A_{n-1,k}(t)\\
&= 2t\sum_{k=1}^{j-1} A_{n-1,k}(t) + (1+t)\sum_{k=j}^{n-j} A_{n-1,k}(t)
+ 2\sum_{k=n+1-j}^{n-1} A_{n-1,k}(t)\\
&= 2\sum_{k=1}^{j-1} \left(tA_{n-1,k}(t) + A_{n-1,n-k}(t)\right) + (1+t)\sum_{k=j}^{\lfloor \frac{n-1}{2}\rfloor}\left(A_{n-1,k}(t) + A_{n-1,n-k}(t)\right) \\
& \qquad \qquad + \begin{cases} (1+t)A_{n-1,n/2} & \mbox{if } 2|n \\
0 & \mbox{if } 2\nmid n.
\end{cases}\\
&= 2\sum_{k=1}^{j-1} \A'_{n-1,k}(t) + (1+t)\sum_{k=j}^{\lfloor n/2 \rfloor }\A_{n-1,k}(t),
\end{align*}
which yields statement (2).

Lastly we consider $\A'_{n,j}(t)$ in terms of Observation \ref{obs:rec}. We have
\begin{align*}
\A'_{n,j}(t) &= tA_{n,j}(t) + A_{n,n+1-j}(t)\\
&= t^2\sum_{k=1}^{j-1} A_{n-1,k}(t) + t\sum_{k=j}^{n-1}A_{n-1,k}(t) + t\sum_{k=1}^{n-j}A_{n-1,k}(t) + \sum_{k=n+1-j}^{n-1}A_{n-1,k}(t)\\
&= t^2\sum_{k=1}^{j-1} A_{n-1,k}(t) + t\sum_{k=1}^{n-j}A_{n-1,n-k}(t) + t\sum_{k=1}^{n-j}A_{n-1,k}(t) + \sum_{k=1}^{j-1}A_{n-1,n-k}(t)\\
&= t\sum_{k=1}^{j-1}\A'_{n-1,k}(t) + t\sum_{k=j}^{n-j}A_{n-1,n-k}(t) + \sum_{k=1}^{j-1}\A'_{n-1,k}(t) +  t\sum_{k=j}^{n-j}A_{n-1,k}(t)\\
&= (t+1)\sum_{k=1}^{j-1}\A'_{n-1,k}(t) + t\sum_{k=j}^{n-j}\left(A_{n-1,k}(t) + A_{n-1,n-k}(t)\right)\\
&=(t+1)\sum_{k=1}^{j-1} \A'_{n-1,k}(t) + 2t\sum_{k=j}^{\lfloor n/2 \rfloor }\A_{n-1,k}(t).
\end{align*}
This confirms (3), completing the proof.
\end{proof}

\begin{remark}
In the case of the ordinary Eulerian polynomials, the coefficients of the $\gamma$-vector have the following interpretation, first described by Foata and Sch\"utzenberger \cite{FS}. Let $\widehat{S}_n = \{ w \in S_n : w_{n-1} < w_n, \mbox{ and if } w_{i-1} > w_i \mbox{ then } w_i < w_{i+1}\}$. Then $\gamma_i = | \{ w \in \widehat{S}_n : d(w) = i\}|$, so that $A_n(t) = \sum \gamma_i t^i(1+t)^{n-1-2i}$. It would be interesting to have a similar combinatorial interpretation for the coefficients of the vectors $\gamma^{(n,j)}$.
\end{remark}

\section{Symmetric $h$-vectors}\label{sec:main}

If $\Delta$ has a symmetric $h$-vector (that is, if $h_i = h_{n-i}$ for all $i$, where $\Delta$ has dimension $n-1$), then from Theorem \ref{thm:BW} and Observation \ref{obs:sym} it follows that $h(\sd(\Delta))$ is symmetric as well.

\begin{cor}[{Brenti and Welker \cite[Corollary 2.6]{BW}}]
If $h(\Delta)$ is symmetric, then $h(\sd(\Delta))$ is symmetric.
\end{cor}

Moreover, the following proposition holds.

\begin{prp}\label{prp:symm}
If $\Delta$ is a boolean complex of dimension $n-1$, with $h(\Delta) = (h_0, h_1,\ldots,h_{n})$ symmetric, then
\[ h_i(\sd(\Delta)) = \sum_{j=0}^{\lfloor n/2\rfloor} (A(n+1,i,j+1) + A(n+1, i , n+1-j)) h_j,\]
and thus
\[ h(\sd(\Delta); t) = \sum_{i=0}^{\lfloor n/2 \rfloor} h_i \A_{n+1,i+1}(t).\] In terms of $\gamma$-vectors,
\[ \gamma(\sd(\Delta)) = \sum_{i=0}^{\lfloor n/2 \rfloor}h_i\gamma^{(n+1,i+1)} .\]
\end{prp}

For example, if $n = 5$ and $h(\Delta) = (h_0,h_1,h_2,h_3=h_2,h_4=h_1,h_5=h_0)$, then
\begin{align*} h(\sd(\Delta))^t &= \left( \begin{array}{ c }
h_0 \\
27h_0 + 18h_1 + 12 h_2 \\
92h_0 + 102 h_1 + 108 h_2 \\
92h_0 +102h_1 + 108 h_2 \\
27h_0 + 18h_1 + 12 h_2 \\
h_0
\end{array} \right)\\
 &=
h_0\left( \begin{array}{ c }
1 \\
5 \\
10 \\
10 \\
5 \\
1
\end{array} \right) +
(22h_0+18h_1+12h_2)\left( \begin{array}{ c }
0 \\
1 \\
3 \\
3 \\
1 \\
0
\end{array} \right)+
(16h_0+48h_1+72h_2)\left( \begin{array}{ c }
0 \\
0 \\
1 \\
1 \\
0 \\
0
\end{array} \right).
\end{align*}
Equivalently, \[ h(\sd(\Delta);t) = h_0\A_{6,1}(t) + h_1\A_{6,2}(t) + h_2\A_{6,3}(t),\] or
\begin{align*}
 \gamma(\sd(\Delta)) &= h_0\gamma^{(6,1)} + h_1\gamma^{(6,2)} + h_2\gamma^{(6,3)}\\
 &=h_0(1,22,16) + h_1(0,18,48) + h_2(0,12,72).
\end{align*}

If $f=(f_0,f_1,\ldots,f_d)$ is a $d$-FFK vector and $f_d\neq 0$, we simply say that $f$ is an FFK-vector, that is, $f$ is the $f$-vector of a balanced complex.
Our goal now is to show that $\gamma(\sd(\Delta))$ is an FFK-vector.

\subsection{Proof of Theorem~\ref{thm:bary}}

Since $\A_{n,1}(t) = (1+t)A_{n-1}(t)$ (Observation \ref{obs:eul}), where $A_{n-1}(t)$ is the usual Eulerian polynomial, we see that $\A_{n,1}(t)$ and $A_{n-1}(t)$ have the same $\gamma$-vector. By \cite[Theorem 6.1 (1)]{NP}, the $\gamma$-vector of $A_{n-1}(t)$ is an FFK-vector, and thus $\gamma^{(n,1)}$ is an FFK-vector for any $n$.
As $A_{n-1}(t)$ has symmetry axis at degree $\lfloor \frac{n}{2} \rfloor-1$, $\gamma^{(n,1)}=(1,f_1,\ldots,f_d)$ where $d=\lfloor \frac{n}{2} \rfloor -1$.
Since $h_0=1$ for any boolean complex $\Delta$, we have (if $\dim \Delta = n-1$): \[ \gamma(\sd(\Delta)) = \gamma^{(n+1,1)} + h_1\gamma^{(n+1,2)} + \cdots \]

Observe that $\gamma^{(n+1,j)}_0 =0$ for all $j>1$. What we will show is that the vectors $h_i\gamma^{(n+1,i+1)}$ are $d$- or $(d+1)$-good for $\gamma^{(n+1,1)}$, in the sense of Definition \ref{def:good}, for all $i\geq 1$. More precisely, we will prove the following.

\begin{prp}\label{prp:good}
Let $\gamma^{(n,1)} = (1,f_1,\ldots,f_d)$, where $d = \lfloor n/2 \rfloor -1$.
\begin{enumerate}
\item If $n$ is even, i.e., $n=2d+2$, then
\begin{enumerate}
 \item $\gamma^{(n,j)}$, $1<j\leq n/2$, is $d$-good for $\gamma^{(n,1)}$, and
 \item $\gamma'^{(n,j)}$, $1\leq j \leq n/2$, is $(d+1)$-good for $\gamma^{(n,1)}$.
\end{enumerate}
\item If $n$ is odd, i.e., $n=2d+3$, then
\begin{enumerate}
\item $\gamma^{(n,j)}$, $1<j\leq (n+1)/2$, is $(d+1)$-good for $\gamma^{(n,1)}$, and
\item $\gamma'^{(n,j)}$, $1\leq j< (n+1)/2$, is $(d+1)$-good for $\gamma^{(n,1)}$.
\end{enumerate}
\end{enumerate}
\end{prp}

In proving Proposition \ref{prp:good}, it will be helpful to recall a definition of \cite{NP} and collect some preliminary results.

Let \[\widehat{S}_n = \{ w \in S_n : w(n-1) < w(n), \mbox{ and if } w(i-1) > w(i) \mbox{ then } w(i) < w(i+1) \}.\] In other words, $\widehat{S}_n$ is the set of permutations in $S_n$ with no double descents and no final descent. We write elements from this set as permutations in one-line notation with bars at the descent positions, or as certain ordered lists of blocks written in increasing order. For example, \[ 235|1479|68, \quad 8|34|19|27|56, \quad 1235|46789\] are elements of $\widehat{S}_9$. Notice that the leftmost block can have one element, but all other blocks have size at least two. Further, if $w = B_1 | \cdots |B_k$, then $\max{B_i} > \min{B_{i+1}}$ for $i=1,\ldots,k-1$. That is, the number of bars in $w$ is $d(w)$.

Define $\Gamma(n)$ to be the simplicial complex whose faces are the elements of $\widehat{S}_n$, with $\dim w = d(w) -1$. We have $w \subseteq v$ if $v$ can be obtained from $w$ by refinement of blocks. Vertices are elements with only one bar. This is a balanced simplicial complex of dimension $\lfloor \frac{n-1}{2} \rfloor -1$.
The color set of a face $w$ is $\col(w) = \{ \lceil i/2 \rceil : w(i)> w(i+1)\}$. A result of \cite{NP} has \[ f(\Gamma(n)) = \gamma^{(n+1,1)}.\] With this interpretation for $\gamma^{(n+1,1)}$, we can prove the following.

\begin{lem}\label{lem:Gineq}
For all $n\geq 1$, $1\leq i \leq \lfloor n/2\rfloor -1$, we have \[(i+1)\gamma^{(n,1)}_i \leq \gamma^{(n+1,1)}_i.\]
\end{lem}

\begin{proof}
Consider any face of $\Gamma(n-1)$ with $i$ bars, say $w = B_1 | \cdots | B_{i+1}$. Then we can associate to $w$ a face of $\Gamma(n)$ by adding the number $n$ to any of the $i+1$ blocks of $w$. If $n$ is inserted at the end of a block, no new descents will be created and hence this insertion takes $(i-1)$-dimensional faces to $(i-1)$-dimensional faces. If $w \neq v$ are in $\Gamma(n-1)$ then it is clear that the set of faces formed from $w$ cannot intersect the set of faces formed from $v$. The inequality follows.
\end{proof}

The next lemma is crucial to later analyses.

\begin{lem}\label{lem:transgood}
Let $n=2d+2$. Then:
\begin{enumerate}
\item If $(0,f)$ is $d$-good for $\gamma^{(n-1,1)}$, then $(0,f)$ is $d$-good for $\gamma^{(n,1)}$.
\item If $(0,f)$ is $(d+1)$-good for $\gamma^{(n,1)}$, then $(0,f)$ is $(d+1)$-good for $\gamma^{(n+1,1)}$.
\end{enumerate}
\end{lem}

\begin{proof}
To prove (1), we consider $f = (f_1,\ldots,f_{d}) = f^{(1)}+\cdots + f^{(k)}$, where each $f^{(j)}=(1,f^{(j)}_1,\cdots,f^{(j)}_{d-1})$ is a $(d-1)$-FFK vector and
$f^{(j)}_i \leq \gamma^{(n-1,1)}_i$ for all $i$ and all $j$. Then by Lemma \ref{lem:Gineq}, we have that for all $1\leq i\leq d-1$
\[ (i+1)f^{(j)}_i \leq (i+1)\gamma^{(n-1,1)}_i \leq \gamma^{(n,1)}_i,\] and thus from Lemma \ref{lem:dgood} we conclude that $(0,f)$ is $d$-good for $\gamma^{(n,1)}$.

We now consider (2). If $(0,f)$ is $(d+1)$-good for $\gamma^{(n,1)}$, then we can write $f=f^{(1)} + \cdots + f^{(k)}$, for some $d$-FFK-vectors $f^{(j)}$ such that $\gamma^{(n,1)}_i \geq f^{(j)}_i$ for all $i$ and $j$. By Lemma \ref{lem:Gineq}
we find $\gamma^{(n+1,1)}$ dominates the $f^{(j)}$ as well, and so by Lemma \ref{lem:cone1}
we conclude that $(0,f)$ is $(d+1)$-good for $\gamma^{(n+1,1)}$ (note that $\gamma^{(n+1,1)}$ has degree $d$).
\end{proof}

In particular, since $(0,\gamma^{(n-1,1)})$ is clearly $d$-good for $\gamma^{(n-1,1)}$, Lemma \ref{lem:transgood} implies $(0,\gamma^{(n-1,1)})$ is $d$-good for $\gamma^{(n,1)}$. Similarly, $(0,\gamma^{(n,1)})$ is $(d+1)$-good for $\gamma^{(n+1,1)}$. We are now ready to prove Proposition \ref{prp:good}.

\begin{proof}[Proof of Proposition \ref{prp:good}.]

We will proceed by induction on $d = \lfloor n/2 \rfloor -1$. If $d=0$, we have:
\[\begin{array}{ c c l c }
\gamma^{(2,1)} = (1) & \gamma'^{(2,1)} = (0,2) & \gamma^{(3,1)} = (1) & \gamma'^{(3,1)} = (0,2) \\
 & & \gamma^{(3,2)} = (0,2) &
\end{array}\]
and the claims are trivially verified. That is, $(0,2)$ is $1$-good for $(1)$. With $d=1$, we have:
\[\begin{array}{ l l l l }
\gamma^{(4,1)} = (1,2) & \gamma'^{(4,1)} = (0,2,4) & \gamma^{(5,1)} = (1,8) & \gamma'^{(5,1)} = (0,2,16) \\
\gamma^{(4,2)} = (0,6) & \gamma'^{(4,2)} = (0,2,4) & \gamma^{(5,2)} = (0,10,8) & \gamma'^{(5,2)} = (0,2,16) \\
 & & \gamma^{(5,3)} = (0,4,8) &
\end{array}\]
From here we see the first instance of part (1a) of the proposition.

Now suppose that the claims of the proposition hold for $d-1$ and we will prove it for $d$.

\noindent\textbf{Case 1 ($\bm{n}$ even).} Let $n=2d+2$ and consider $\gamma^{(n,j)}$ for some $1<j\leq n/2$. We wish to show that $\gamma^{(n,j)}$ is $d$-good for $\gamma^{(n,1)}$. By Lemma \ref{lem:rec}, we have \[ \gamma^{(n,j)} = 2\sum_{k=1}^{j-1} \gamma'^{(n-1,k)} + \sum_{k=j}^{n/2 } \gamma^{(n-1,k)}.\] Since $n-1=2(d-1)+3$ is odd, our induction hypothesis has each summand $d$-good for $\gamma^{(n-1,1)}$. The sum of $d$-good vectors is again $d$-good, so $\gamma^{(n,j)}$ is $d$-good for $\gamma^{(n-1,1)}$. By Lemma \ref{lem:transgood} (1), we conclude that $\gamma^{(n,j)}$ is also $d$-good for $\gamma^{(n,1)}$, proving part (1a).

Now we wish to show $\gamma'^{(n,j)}$, with $1\leq j\leq n/2$, is $(d+1)$-good for $\gamma^{(n,1)}$. From Lemma \ref{lem:rec} we have
\[
 \gamma'^{(n,j)} = \sum_{k=1}^{j-1} \gamma'^{(n-1,k)} + 2\sum_{k=j}^{n/2 } (0, \gamma^{(n-1,k)}).
 \]
In the degenerate case $j=1$, this gives $\gamma'^{(n,1)} = (0,2\gamma^{(n,1)})$, which is clearly $(d+1)$-good for $\gamma^{(n,1)}$. If $j>1$ we can rewrite this as:
\[ \gamma'^{(n,j)} = \sum_{k=2}^{j-1} \gamma'^{(n-1,k)} + (0, 2f),\] where \[ f = \gamma^{(n-1,1)} + \sum_{k=j}^{\lfloor n/2 \rfloor} \gamma^{(n-1,k)}.\] All the terms $\gamma'^{(n-1,k)}$ in the left summation are $d$-good for $\gamma^{(n-1,1)}$ by our induction hypothesis. By Lemma \ref{lem:transgood} (1), we have that these terms are $d$-good for $\gamma^{(n,1)}$. Also, by our induction hypothesis, $f=(1,f_1,\ldots,f_{d-1})$ is an FFK-vector, and as \[ \gamma^{(n,1)} = \sum_{k=1}^{n/2} \gamma^{(n-1,k)},\] we see that $\gamma^{(n,1)}$ dominates $f$. Thus, by Lemma \ref{lem:cone1}, we see that $(0,2f)$ is $(d+1)$-good for $\gamma^{(n,1)}$. The sum of $d$- and $(d+1)$-good vectors is $(d+1)$-good by Observation \ref{obs:dd+1}, so we conclude that $\gamma'^{(n,j)}$ is $(d+1)$-good for $\gamma^{(n,j)}$, proving 1(b) as desired.

\noindent\textbf{Case 2 ($\bm{n}$ odd).} Let $n=2d+3$, and consider $\gamma^{(n,j)}$ with $1<j\leq (n+1)/2$. We wish to show $\gamma^{(n,j)}$ is $(d+1)$-good for $\gamma^{(n,1)}$. As a special case, if $j=(n+1)/2$, then Lemma \ref{lem:rec} (1) has \[ \gamma^{(n,(n+1)/2)} = \sum_{k=1}^{(n-1)/2} \gamma'^{(n-1,k)}.\] As shown in Case 1, each term in the sum is $(d+1)$-good for $\gamma^{(n-1,1)}$, and so by Lemma \ref{lem:transgood}, we conclude that $\gamma^{(n,(n+1)/2)}$ is $(d+1)$-good for $\gamma^{(n,1)}$.

Now suppose $1<j<(n+1)/2$. By Lemma \ref{lem:rec} part (2), \[ \gamma^{(n,j)} = 2\sum_{k=1}^{j-1} \gamma'^{(n-1,k)} + \sum_{k=j}^{(n-1)/2} \gamma^{(n-1,k)}.\] It was shown in Case 1 that all the terms in the left summation are $(d+1)$-good for $\gamma^{(n-1,1)}$ and the terms in the right summation are $d$-good for $\gamma^{(n-1,1)}$. Thus, $\gamma^{(n,j)}$ is $(d+1)$-good for $\gamma^{(n-1,1)}$, and by Lemma \ref{lem:transgood} $\gamma^{(n,j)}$ is $(d+1)$-good for $\gamma^{(n,1)}$. This proves part (2a).

For part (2b), our analysis is identical to Case (1b). That is, in the $j=1$ case we get $\gamma'^{(n,1)} = (0,2\gamma^{(n-1,1)})$, which is clearly $(d+1)$-good for $\gamma^{(n-1,1)}$, and hence for $\gamma^{(n,1)}$. If $1<j<(n+1)/2$ we have \[ \gamma'^{(n,j)} = \sum_{k=2}^{j-1} \gamma'^{(n-1,k)} + (0, 2f),\] and it follows that all terms involved are $(d+1)$-good for $\gamma^{(n-1,1)}$; hence, for $\gamma^{(n,1)}$. This proves part (2b), and completes the proof of the proposition.
\end{proof}

As discussed above, we have now proved Theorem \ref{thm:bary}.

In particular, as barycentric subdivisions are flag, and the operation of taking barycentric subdivisions leaves the topology of the underlying space unchanged, we confirm the following case of Conjecture 1.4 of \cite{NP}. (Indeed, the $h$-vector of a homology sphere is symmetric and nonnegative \cite{StanleyUBC}.)

\begin{cor}\label{cor:homologyFFK}
If $\Delta$ is a homology sphere, then $\gamma(\sd(\Delta))$ is an FFK-vector.  In other words, $\gamma(\sd(\Delta))$ is the $f$-vector of a balanced simplicial complex.
\end{cor}

Frohmader \cite{Frohmader} proved that the $f$-vectors of flag simplicial complexes form a (proper) subset of the $f$-vectors of balanced complexes. In \cite{NP}, all the $\gamma$-vectors in question were shown to be the $f$-vectors of flag complexes. This suggests the following problem, for which a positive answer would imply stronger conditions on $\gamma(\sd(\Delta))$ than Corollary \ref{cor:homologyFFK}.

\begin{prob}
Let $\Delta$ be a homology sphere. Is $\gamma(\sd(\Delta))$ the $f$-vector of a flag simplicial complex?
\end{prob}

\begin{remark}
It would be interesting if one could explicitly construct a simplicial complex $\Gamma(\sd(\Delta))$ such that $f(\Gamma) = \gamma(\sd(\Delta))$, as in Section 4 of \cite{NP}. Such a construction may follow from a ``nice" combinatorial interpretation for the entries of the vectors $\gamma^{(n,j)}$.
\end{remark}

\section{Implications for $h$- and $g$-vectors}\label{sec:implications}

It is easy to see that $\gamma$-nonnegativity implies $h$-nonnegativity, as well as $g$-nonnegativity. Moreover, if the polynomial $\gamma(t)$ has only real roots then $h(t)$ and $g(t)$ have only real roots as well. (This follows from the fact that $h(t)$ and $g(t)$ are obtained by totally nonnegative linear transformations of $\gamma(t)$. See work of Br\"and\'en \cite{B} for an introduction to these type of results.) In the context of this work, it is natural to ask whether the fact that $\gamma$ is the $f$-vector of a simplicial complex implies that the same is true of the $h$- or $g$-vector. We show here that the answer is affirmative.

\begin{remark}
Brenti and Welker show \cite[Theorem 3.1]{BW} that $h(\sd(\Delta); t)$ has only real roots. This is another way to prove the Charney-Davis conjecture for barycentric subdivisions. See, for example, the discussion after Example 3.7 in \cite{BW}. We have not attempted to investigate whether it is true that $\gamma(\sd(\Delta);t)$ has only real roots, but if true, this would imply Brenti and Welker's Theorem 3.1 as well. However, as already mentioned in the introduction, our Theorem \ref{thm:bary} implies Gal's conjecture for barycentric subdivisions, which in turn implies the Charney-Davis conjecture in this case.
\end{remark}

Suppose throughout this section that, respectively, $h=(h_0,h_1,\ldots,h_d)$ (with $h_i = h_{d-i}$), $g=(g_0,g_1,\ldots, g_{\lfloor d/2\rfloor})$, and $\gamma=(\gamma_0,\gamma_1,\ldots, \gamma_{\lfloor d/2 \rfloor})$ are the $h$-, $g$-, and $\gamma$-vectors of some $(d-1)$-dimensional boolean complex with a nonnegative and symmetric $h$-vector.

First, we have the following observation for relating the $h$- and $g$-vectors to the $\gamma$-vector.

\begin{obs}\label{obs:trans}
We have, for $0\leq i \leq \lfloor d/2 \rfloor$:
\[ h_i = h_{d-i} = \sum_{0\leq j \leq i} \gamma_j \binom{d-2j}{i-j},\] and \[g_i = \sum_{0\leq j \leq i} \gamma_j\left( \binom{d-2j}{i-j}-\binom{d-2j}{i-j-1} \right).\]
\end{obs}

We can transform a $\gamma$-vector into its corresponding $h$- or $g$-vector with the following transformations:
\[ A:=\left[ \binom{d-2j}{i-j} \right]_{0\leq i,j \leq d/2}, \quad B:=\left[ \binom{d-2j}{i-j}-\binom{d-2j}{i-j-1}\right]_{0 \leq i,j \leq d/2}.\] We have $A\gamma = (h_0, \ldots,h_{\lfloor d/2 \rfloor})$ (construct the rest of $h$ by symmetry) and $B\gamma = g$. We remark that matrices $A$ and $B$ are \emph{totally nonnegative}, i.e., the determinant of every minor of these matrices are nonnegative. (It is straightforward to prove this with a ``Lindstrom-Gessel-Viennot"-type argument involving planar networks.)

Notice that at the extremes we have $g_0 = \gamma_0$ and \[ g_{\lfloor d/2 \rfloor} = \sum_{0\leq j \leq d/2} \gamma_j C_{\lceil d/2\rceil -j},\] where $C_r = \binom{2r}{r} - \binom{2r}{r-1} = \binom{2r}{r}/(r+1)$ is a Catalan number.

\begin{prp}\label{prp:fh}
If $\gamma$ is an $f$-vector of a simplicial complex then $h$ is an $f$-vector of a simplicial complex.
\end{prp}

\begin{proof}
Let $\F(\gamma)$ denote the standard compressed complex for $\gamma$ (take the definition for $\F_d(\gamma)$ and let $d \to \infty$). Then
$\F(\gamma)$ is a simiplicial complex. Let \[ \Delta = \{ F \cup G : F \in \F(\gamma), G \in 2^{[d - 2|F|]} \},\] where $2^{[k]}$ denotes a $(k-1)$-simplex on a vertex set $\{1,2,\ldots,k\}$ disjoint from the vertex set of $\F(\gamma)$. It is straightforward to see that $\Delta$ is in fact a simplicial complex. Indeed, if $\bar F = F \cup G$ is a face of $\Delta$, then all subsets of $\bar F$ are of the form $\bar H = F' \cup G'$, where $F' \subseteq F$ and $G' \subseteq G$. As $F' \in \F(\gamma)$ and $G' \in 2^{[k]}$ for all $k \geq |G'|$, we see that $\bar H \in \Delta$.

We compute
\begin{align*}
f_i(\Delta) = \sum_{|F\cup G| =i} 1 = \sum_{0\leq j \leq i} \sum_{|F| = j} \sum_{|G| = i-j} 1 &= \sum_{0\leq j\leq i} f_j(\F(\gamma))\cdot f_{i-j}(2^{[d-2j]}),\\
&=\sum_{0\leq j \leq i} \gamma_j \binom{d-2j}{i-j}.
\end{align*}

By Observation \ref{obs:trans}, we conclude that $h_i = f_i(\Delta)$, completing the proof.
\end{proof}

The corresponding result for $g$-vectors follows the same approach, but first requires the construction of an auxillary complex $\B(k)$ defined below.

A \emph{ballot path} is a lattice path that starts at $(0,0)$, takes steps of the form $(0,1)$ (north) and $(1,0)$ (east), and does not go above the line $y=x$ (but it can touch the line). We write ballot paths as words $p$ on $\{N,E\}$ such that any initial subword of $p$ has at least as many letters $E$ as letters $N$.
Let $B(k)=\{\mbox{ballot paths of length } k\}$, and for $p =p_1\cdots p_k \in B(k)$, let $S(p) = \{ k+1-i : p_i = N\}$ denote the set of positions of letters $N$ in $p$ (read from right to left). For example, the ballot path $p=ENEENEE$ has $S(p)=\{3,6\}$.

Let $\B(k) = \{ S(p) : p \in B(k) \}$ denote the set of sets of north steps for ballot paths of length $k$.

\begin{prp}
For any $k \geq 0$, $\B(k)$ is a simplicial complex on vertex set $[k-1]$. Moreover, $\dim \B(k) = \lfloor k/2 \rfloor -1$, and $f(\B(k)) = (f_0, f_1, \ldots, f_{\lfloor k/2 \rfloor})$ is given by \[ f_i(\B(k)) = \binom{k}{i} - \binom{k}{i-1}.\]
\end{prp}

\begin{proof}
To see that $\B(k)$ is a simplicial complex, simply observe that if $p$ is ballot path, changing any north step to an east step results in a path that remains below $y=x$. That the vertex set is $\{1,2,\ldots,k-1\}$, follows from observing that the first step of a ballot path must be east.

The dimension of a face is simply one less than the number of north steps taken in the corresponding path. Thus,
\begin{align*}
f_i(\B(k)) &= |\{ p \in B(k) : |S(p)| = i\}|\\
 &= |\{\mbox{lattice paths from $(0,0)$ to $(k-i,i)$ not surpassing the line $y=x$}\}|\\
 &= \binom{k}{i} - \binom{k}{i-1}.
\end{align*} The final equality is a straightforward counting argument described, e.g., in \cite[Exercise 6.20]{St}. (Show there are $\binom{k}{i-1}$ paths from $(0,0)$ to $(k-i,i)$ that go above $y=x$.)

To complete the proof, observe that a maximal face corresponds to a path from $(0,0)$ to $(k/2,k/2)$ if $k$ is even, or from $(0,0)$ to $(\frac{k+1}{2}, \frac{k-1}{2})$ if $k$ is odd. Thus, $\dim \B(k) = \lfloor k/2 \rfloor -1$.
\end{proof}

We remark that as $f_{\lfloor k/2 \rfloor} = \binom{k}{\lfloor k/2 \rfloor} - \binom{k}{\lfloor k/2 \rfloor -1} = C_{\lceil k/2 \rceil}$, the number of facets of $\B(k)$ is a Catalan number.

\begin{prp}\label{prp:fg}
If $\gamma$ is an $f$-vector of a simplicial complex then $g$ is an $f$-vector of a simplicial complex.
\end{prp}

\begin{proof}
Similarly to the proof of Proposition \ref{prp:fh}, let $\F(\gamma)$ denote the standard compressed complex for $\gamma$. Now, let \[ \Delta = \{ F \cup G : F \in \F(\gamma), G \in \B(d - 2|F|) \},\] where we take the vertex set $\{1,2,\ldots,k-1\}$ for $\B(k)$ to be disjoint from the vertex set of $\F(\gamma)$. To see that $\Delta$ is in fact a simplicial complex suppose $\bar F = F \cup G$ is a face of $\Delta$. All subsets of $\bar F$ are of the form $\bar H = F' \cup G'$, where $F' \subseteq F$ and $G' \subseteq G$. That $F' \in \F(\gamma)$ is obvious. To prove $\bar H \in \Delta$, it remains to show that $G' \in \B(k)$ for all $k \geq |G'|$. To see this, consider path $p \in B(|G'|)$ such that $S(p)=G'$. For any $i\geq 0$, we can extend this path to a path $p'$ in $B(|G'|+i)$ by prepending $i$ east steps to $p$. In other words, \[ p' = \underbrace{EE\cdots E}_i p \in B(|G'|+i)\] and $S(p') = S(p) = G'$.

We conclude
\begin{align*}
f_i(\Delta) = \sum_{|F\cup G| =i} 1 = \sum_{0\leq j \leq i} \sum_{|F| = j} \sum_{|G| = i-j} 1 &= \sum_{0\leq j\leq i} f_j(\F(\gamma))\cdot f_{i-j}(\B(d-2j)),\\
&=\sum_{0\leq j \leq i} \gamma_j \left( \binom{d-2j}{i-j} - \binom{d-2j}{i-j-1}\right).
\end{align*}

By Observation \ref{obs:trans}, we conclude that $g_i = f_i(\Delta)$, completing the proof.
\end{proof}

Considering the results of this paper and those of \cite[Theorem 6.1]{NP}, we have the following consequences of Propositions \ref{prp:fh} and \ref{prp:fg}.

\begin{cor}\label{thm:gh}
We have that $h(\Delta)$ and $g(\Delta)$ are $f$-vectors of simplicial complexes if:
\begin{enumerate}
 \item[(a)] $\Delta$ is the barycentric subdivision of a boolean complex with symmetric and nonnegative $h$-vector (e.g., $\Delta$ is the barycentric subdivision of a simplicial sphere).
 \item[(b)] $\Delta$ is a Coxeter complex.
 \item[(c)] $\Delta$ is the simplicial complex dual to an associahedron.
 \item[(d)] $\Delta$ is the simplicial complex dual to a cyclohedron.
 \item[(e)] $\Delta$ is a flag homology sphere with $\gamma_1(\Delta)\leq 3$.
\end{enumerate}
\end{cor}

We remark that it is known that $h(\Delta)$ is the $f$-vector of a balanced simplicial complex whenever $\Delta$ is a balanced Cohen-Macaulay complex \cite[Theorem 4.6]{Stanley}. As mentioned in the introduction, recent work of Murai \cite{MuraiBr} proves Corollary \ref{thm:gh}(a) (for $g$-vectors) directly, building on the approach of Brenti and Welker.
That the $g$-vector of the barycentric subdivision of a Cohen-Macaulay complex is an $M$-sequence was proved earlier by algebraic means  \cite{KN}.

One may ask whether the results of Propositions \ref{prp:fh} and \ref{prp:fg} can be strengthened to $f$-vectors of \emph{balanced} simplicial complexes. 
As we will show below, the answer is Yes for the $h$-vector and No for the $g$-vector.

First let us remark that for flag spheres of dimension at most $3$, it is known that the $\gamma$-vector is an FFK-vector. Then a simple computation shows that the $g$-vector is an FFK-vector as well. However, this does not hold in higher dimensions. 

\begin{example}
Let $\Delta$ be the triangulation of the $4$-sphere obtained by taking the suspension of the join of two $(k+2)$-gons, where $k\geq 2$. Then $\Delta$ is a flag sphere with \[h(\Delta) = (1,2k+1,(k+1)^2+1, (k+1)^2+1, 2k+1,1).\] Its $\gamma$-vector is  $(1,2(k-2), (k-2)^2)$, which is $2$-FFK, while $g(\Delta)=(1,2k, k^2+1)$, which is not $2$-FFK. 
\end{example}
(We see that $\gamma$ is the $f$-vector the complete bipartite graph on $2(k-2)$ vertices, while there can be no bipartite graph with $2k$ vertices and more than $k^2$ edges.)

\begin{prp}\label{prp:FFKh}
If $\gamma$ is the $f$-vector of a balanced simplicial complex then $h$ is the $f$-vector of a balanced simplicial complex.
\end{prp}
\begin{proof}
Let $\gamma$ be the $f$-vector of a balanced complex. In particular, $\gamma$ is $d$-FKK (as it is $\lfloor \frac{d}{2}\rfloor$-FFK).
Let $\Delta$ be a $d$-colorable simplicial complex with
$f(\Delta)=\gamma$ on a vertex set $V$, and with a proper coloring $c: V \rightarrow [d]$  (assume $V\cap [d]=\emptyset$).
For $F \in \Delta$, let $d(F)$ be the set of the first $|F|$ integers in $[d] -c(F)$.
Consider the following simplicial complex

\[\Gamma = \{ F \cup G:\ F \in \Delta, G \subset 2^{[d] - (c(F)\cup d(F))}\}.\]

Then $\Gamma$ is a $d$-colorable simplicial complex with $f(\Gamma)=h$. Indeed, each face $F\cup G$ of $\Gamma$ has distinctly colored vertices by construction. To see that $\Gamma$ is indeed a simplicial complex, note that $F' \subseteq F\subseteq V$ implies $c(F') \subseteq c(F)$, so if $F\cup G\in \Gamma$ then $G\subseteq [d] - (c(F)\cup d(F)) \subseteq [d] - (c(F')\cup d(F'))$, hence $F'\cup G \in \Gamma$.
\end{proof}

Thus, a consequence of Conjecture \ref{conj:flagKrK} would be the following.

\begin{conj}
If $\Delta$ is a flag homology sphere then $g(\Delta)$ is the $f$-vector of a simplicial complex and $h(\Delta)$ is the $f$-vector of a balanced simplicial complex.
\end{conj}

\vspace{0.1 in}
{\it{Acknowledgement.}} 
We thank Satoshi Murai for pointing out how to modify the proof of Proposition \ref{prp:fh} in order to prove Proposition \ref{prp:FFKh}.

\end{document}